\documentclass[12pt, a4paper]{amsart}
\input epsf
\usepackage{hyperref}
\usepackage{setspace}
\usepackage[margin=1in]{geometry}
\numberwithin{equation}{section}

\newtheorem{thm}{Theorem}[section]
\newtheorem{lemma}[thm]{Lemma}
\newtheorem{prop}[thm]{Proposition}
\newtheorem{cor}[thm]{Corollary}

\newcommand{\comment}[1]{}

\newcommand{\R}{{\mathbb R}}

\renewcommand{\L}{{\mathbb L}}
\newcommand{\prob}{\nu}
\newcommand {\E}[1]{\int#1\mbox{d}\nu}
\newcommand{\abs}[1]{\left|#1\right|}
\newcommand{\PHI}[1]{\Phi\circ f^{#1}}

\title[Extreme value theory for skew-product dynamical systems]{Extreme Value Distributions for some classes of Non-Uniformly Partially Hyperbolic Dynamical Systems}
{\author[Chinmaya Gupta]{Chinmaya Gupta}
\address[Chinmaya Gupta]{Department of Mathematics, University of Houston, 4800 Calhoun Rd., Houston, TX, 77204, USA}
\email{ccgupta@math.uh.edu}
\urladdr{http://math.uh.edu/~ccgupta}
}
\date{\today}

\begin{document}

\maketitle

\begin{abstract}
In this note, we obtain verifiable sufficient conditions for the extreme value distribution for a certain class of skew product extensions of non-uniformly hyperbolic base maps. We show that these conditions, formulated in terms of the decay of correlations on the product system and the measure of rapidly returning points on the base, lead to a distribution for the maximum of $\Phi(p) = -\log(d(p, p_0))$ that is of the first type. In particular, we establish the Type I distribution for $S^1$ extensions of piecewise $C^2$ uniformly expanding maps of the interval, non-uniformly expanding maps of the interval modeled by a Young Tower, and a skew product extension of a uniformly expanding map with a curve of neutral points.

\end{abstract}

\section{Introduction}
Suppose that $\{X_i\}$ is a stochastic process and we define a stochastic
process $\{M_n\}$ of successive maxima by $M_n = \max \{X_1, \ldots , X_n \}$. Extreme value
theory is concerned with the limiting distribution of $\{M_n\}$ under
linear scalings $a_n (M_n -b_n)$ defined by constants $a_n>0, b_n \in \R$.
In the 
i.i.d case there is a well-developed theory~\cite{Leadbetter,Galambos,Resnick} 
and  it is known that there are
only three possible non-degenerate distributions under linear scaling i.e.,
if $\{X_i \}$ is i.i.d., $a_n>0,b_n \in \R$ are scaling constants and $G(x)$
is a non-degenerate distribution defined by
\[
\lim_{n\rightarrow \infty} P( a_n (M_n - b_n)\le x)=G(x)
\]
then $G(x)$ has one of three possible forms (up to scale and location
changes), which we call extreme type
  distributions.
  
%To determine which, if any, of the extremal distributions are satisfied by
%the time-series of observations on dynamical systems is not as
%straightforward. The dependence conditions of Loynes~\cite{Loynes} and
%Leadbetter et al~\cite{Leadbetter} are phrased in terms of the mixing
%properties of $\sigma$-algebras generated by observations and not in terms
%of the usual mixing conditions from dynamical systems theory, and so are
%not immediately applicable. There has been some limited progress on extreme
%value theory for dynamical systems. Haiman~\cite{Haiman} showed that the
%observation $\Phi(x)=x$ on the doubling map is in the domain of attraction
%of a Type III distribution by giving a precise formula for $P(M_n \le
%1-2^{-k})$. 

Collet \cite{collet} studied the return time statistics to a point $x_0$ in the phase
space of a one-dimensional non-uniformly expanding map modeled by a Young
Tower with exponential decay of correlations. He notes that his work can be
interpreted in terms of extreme value statistics for such systems.  Collet
showed that the function $F(x) = -\log d(x, x_0)$ on the systems he
considered displays Type I extreme value statistics for $\mu$ a.e. $x_0$.
Freitas and Freitas \cite{Freitas} showed the corresponding result for these maps when
$x_0$ is taken to be the critical point $c$ or critical value $f(c)$.
Freitas, Freitas and Todd \cite{todd} investigated the link between extreme
value statistics and return time statistics, and  showed that
any multimodal map with an absolutely continuous invariant measure displays
either Type I, II or III extreme value statistics.    This
result required no knowledge of the decay of correlations for these maps.
They also proved that for these systems  the excedance point process
converges to
a Poisson  process. Dolgopyat \cite[Theorem 8]{dolgopyat} has proved Poisson limit laws for the return time
statistics of visits to a scaled neighborhood of a measure theoretically
generic point in uniformly partially hyperbolic systems with exponential
decay of correlations for $C^k$ functions. He also gives distributional
limits for periodic orbits, but again exponential decay is required and
uniform partial hyperbolicity is assumed.

In this note, using arguments based on  Collet's and recent work by Gou\"ezel~\cite{gouezel} on the rate of decay of correlations for compact group extensions of 
non-uniformly expanding maps we establish, to our knowledge, the first extreme value theory (or return time statistics)  for non-uniformly partially hyperbolic systems.
Our main result is Theorem~\ref{thm.main} which gives verifiable conditions on the base transformation and a sufficient (polynomial) rate of decay of correlations for 
a Type I   extreme value distribution to hold for $\Phi(p)=-\log d(p,p_0)$  for $\mu \times \lambda_Y$ a.e. $p_0=(x_0,\theta_0)\in X\times Y$.
This results in the extreme value statistics for observations of a certain degree of regularity with maxima at such points $p_0$. The sufficient conditions of 
Theorem~\ref{thm.main} are verifiable for a residual set of H\"{o}lder  $S^1$-cocyles over certain classes of maps  recorded in Corollary~\ref{cor.res}.
The maps in this category include piecewise $C^2$ uniformly expanding maps and non-uniformly expanding maps with finite 
derivative which may be modeled by a Young Tower with exponential return time tails (such as logistic or unimodal maps, including the class
studied by Collet). 
We also verify, in section \ref{skew-prod}, that Gou\"ezel's map  satisfies hypotheses of our theorem and hence our results also apply to this map. A key role in our verification is played by results due to Gou\"ezel~\cite{gouezel}, on rates of decay of correlations  for $S^1$ extensions of non-uniformly partially hyperbolic systems.

Further, we verify the conditions on the base transformation for a class of intermittent like maps, including the Liverani-Saussol-Vaienti
map. Unfortunately, the rate of decay of correlations of H\"{o}lder observations on compact group extensions of such systems is not known.
Nevertheless we give a sufficient decay rate to ensure Type I extreme value statistics for $-\log d(p, p_0)$ for $\mu\times\lambda_Y$ a.e. $p_0$.  We believe it plausible that 
for sufficiently small $0<\omega<1$, where the germ of the indifferent fixed point is $x\to x+x^{1+\omega}$, this decay rate holds and will 
be proven to hold. We also verify all but one of the hypotheses of our theorem for the Viana Map. The hypothesis that fails concerns the density of the absolutely continuous invariant measure. It is not known whether the density belongs to $\mathbb{L}^{1+\delta}(\lambda)$ for any $\delta>0$.

\section{Framework of the problem}

Suppose  that $Y$ is  a compact, connected  $M$-dimensional manifold with %left invariant 
metric $d_Y$ and  $X$  is a
compact  $N$-dimensional %Riemannian 
manifold   with %Euclidean  
metric  $d_X$. We let $D = M+N$ and define a metric on $X\times Y$ by
\begin{equation}
 d((x_1, \theta_1), (x_2, \theta_2)) = \sqrt{d_X(x_1, x_2)^2 + d_Y(\theta_1 , \theta_2)^2}.
\end{equation}
We denote  the Lebesgue measure on $X$  by $\lambda_X$,  the Lebesgue measure on $Y$  by $\lambda_Y$
and the product measure on $X\times Y$ by $\lambda =\lambda_X \times \lambda_Y$.

We will call a function $\Upsilon: X\times Y \to \R$ H\"older continuous of exponent $\zeta$ if there exists some constant $K$ such that $$\abs{\Upsilon(x) - \Upsilon(y)}\le Kd((x_1, \theta_1),(x_2, \theta_2))^{\zeta}$$ for all $(x_1, \theta_1)$ and $(x_2, \theta_2)$ in $X\times Y$. We define the $C^\zeta$ norm of $\Upsilon$ as $$\|\Upsilon\|_{C^\zeta} = \sup_{(x, \theta)\in X\times Y}\abs{\Upsilon(x, \theta)}+\sup_{\substack{(x, \theta), (y,\rho)\in X\times Y \\ (x, \theta)\ne (y, \rho)}}\frac{\abs{\Upsilon(x, \theta) - \Upsilon(y, \rho)}}{d((x, \theta),(y, \rho))^\zeta}.$$

If $T: X\to X$ is a measurable transformation and $u:X\times Y\rightarrow Y$ a measurable function,  then we may define $f$, the {\bf $Y$-skew-extension of $T$ by $u$ }by, $$f:X\times Y\rightarrow X\times Y$$ 
\begin{equation}
 f(x, \theta) = (Tx,  u(x, \theta)).
\end{equation}

We assume that $T:X \to X$ has an ergodic invariant measure $\mu_X$ with support $X$. %and density $h(x) =\frac{d\mu}{d\lambda_X}\in L^{p}(\lambda_X)$ for some $p>1$.
We further assume that $f: X\times Y \to X\times Y$ preserves an invariant probability measure $\nu$, which  has density $H\in \L^{1}(\mu_X\times\lambda_Y)$  and $H$ is locally $\L^p(\lambda)$ for some $p>1$. 

 We are interested in the extreme value statistics of
 observations which are maximized at a  unique point $(x_0, \theta_0)$.  For the  given  point $(x_{0}, \theta_{0})$
 we define  a function $\Phi$ on $X\times Y$ by
  \[
  \Phi(x, \theta) = -\log d((x, \theta),(x_0, \theta_0))
  \]
  (here the dependence on $(x_0,\theta_0)$ is omitted for notational 
  simplicity).  For a given  $v\in\R$ we define  $u_n = v+\frac{1}{D}\log n$ and  denote by $Z_n$ (more precisely $Z_n^{(x_0,\theta_0)}$) the random variable 
  \[
  Z_n^{(x_0,\theta_0)} =\max(\Phi, \Phi\circ f,\dots ,\Phi\circ f^n).
  \]

  % Let $\beta>0, 0<\gamma<1, a>\frac{\gamma}{\beta}$ and $\gamma' = \frac{\gamma}{a}$. This amounts to asking that $\gamma'<\beta$. 
We will prove the following result:

\begin{thm}\label{thm.main}

Assume that the density $H \in \L^{1+\delta}(\lambda) $(locally) for some $\delta>0$. Let $\kappa>1$ be conjugate to $1+\delta$. Further, assume
\begin{itemize}
\item[(a)]  that there exist constants $C_1>0, \beta>0$ and an increasing function $g(n)\approx n^{D\gamma'}$ (with $0<\gamma'<\frac{\beta}{D}$) such that if 
\[
E_n^X := \left\{x \in X\colon d_X(T^jx, x)<\frac{1}{n}\mbox { for some } j \in \{1,2\dots g(n)\}\right\}
\]
then $\mu_X(E_n^X)<\frac{C_1}{n^\beta}$ 
\item[(b)] that there exists $0<\hat\alpha\le 1$ such that, for all H\"older continuous functions $\Upsilon$ with H\"older exponent $\hat\alpha$,  and $\Psi \in \L^\infty(\nu)$,
\begin{equation}\label{*}
 ~\abs{\int \Psi\circ f^j\Upsilon \mbox{d}\nu - \int\Upsilon\mbox{d}\nu\int\Psi\mbox{d}\nu}\le C_2 \Theta (j) \|\Psi\|_\infty\| \Upsilon\|_{C^{\hat\alpha}} \end{equation}
where $\Theta(j)\le j^{-\alpha}$ and  $\alpha>\frac{\frac{1}{D}\left(1+D\kappa\left(\frac{3}{2}-\frac{1}{\kappa}\right)\right)+\frac{3}{2}}{\min\left\{\gamma', \frac{1}{2}\right\}}$.

\end{itemize}

Then for  $\nu$  a.e. $(x_0,\theta_0)$ and for every  $v\in\R$,  
\begin{equation}
\lim_{n\rightarrow\infty}\nu\left(Z_n^{(x_0, \theta_0)}<u_n\right) = e^{H(x_0, \theta_0)e^{-Dv}}
\end{equation}
\end{thm}

While we will prove theorem \ref{thm.main} for an arbitrary fiber $Y$ that is a compact connected $M$-dimensional manifold, our corollaries will involve the special cases $Y = S^1$ and $Y=[0, 1]$. This is because condition $(b)$ of theorem \ref{thm.main} requires a decay of correlations to hold and we only consider examples for which this decay is known to hold. Further, note that we require $0<\hat\alpha\le1$. This is because for the proof of Lemma \ref{decaycorr}, we need $(b)$ of the above theorem to hold for Lipschitz continuous functions having compact support.

We will make the following definitions: A set will be called {\bf residual} if its complement can be written as a countable union of nowhere dense sets. A $C^r$ {\bf cocycle} $h$ on an interval $I$ into a group $Y$  will be defined as a $C^r$ map $h:I\to Y$. If $h$ is a cocycle, the skew-extension $f$ will be defined as $f(x, \theta) = (Tx, \theta+h(x))$. 

We now state the corollaries to the above theorem (see Section~\ref{sec-applications}).
\begin{cor}\label{cor.res}
If  $Y=S^1$ and $T$ is one of the following transformations:
\begin{itemize}
\item[(a)] a piecewise $C^2$ uniformly expanding map $T:I\to I$ of an interval $I$.
\item[(b)] a one-dimensional non-uniformly expanding  map $T:I\to I$ of an interval $I$ with bounded derivative and modeled by a Young Tower with exponential decay of correlations
\end{itemize}
then for a residual set of  H\"{o}lder cocycles  $h:I\to S^1$, for $\mu_X\times \lambda_Y$ a.e. $(x_0,\theta_0)$ and for all $v\in \R$,   
\begin{equation}
\lim_{n\rightarrow\infty}\nu\left(Z_n^{(x_0, \theta_0)}<u_n\right) = e^{H(x_0, \theta_0)e^{-Dv}}.
\end{equation}
\end{cor}

\begin{cor}\label{cor.gouezel.map}
 Let $T:S^1\times [0, 1]\to S^1\times[0, 1]$ be the map $T(\omega, x) = (4\omega, T_{\alpha(\omega)}(x))$ where the maps $\alpha$ and $T_\alpha$, an intermittent type map, are as defined in Section \ref{skew-prod}. Suppose that $$\alpha_{\max}<\frac{\min\{\gamma', \frac{1}{2}\}}{\min\{\gamma', \frac{1}{2}\}+\frac{1}{D}\left(1+\frac{D}{2}\right)+\frac{3}{2}}.$$ Then for $\nu$ a.e. $(x_0, \theta_0)$ and for each $v \in \R$,  
\begin{equation}
 \lim_{n\to\infty}\nu\left(Z_n^{(x_0, \theta_0)}<u_n\right) = e^{H(x_0, \theta_0)e^{-Dv}}.
\end{equation}

\end{cor}

There are other important classes of maps such as $Y$ extensions of Manneville-Pommeau type maps (for a compact connected Lie group $Y$, for instance, $Y=S^1$) and the Viana type maps that satisfy most, but not all, of our hypotheses. It is not known for $S^1$ extensions of Manneville-Pommeau type maps whether a sufficiently high polynomial rate of decay satisfying our hypotheses holds. Similarly, for the Viana map, all of our hypotheses are satisfied except we do not know whether the density of the invariant measure is locally $\L^p$ for some $p>1$. A further discussion of these maps may be found in Section \ref{sec-applications}.

\section{Preliminaries}

For the rest of the article, we will refer to the function $f^0$ as the identity function and $\chi_A$ as the characteristic function for $A$. Upper-case greek letters, such as $\Phi$ and $\Psi$, will usually denote functions,  while lower-case letters, such as $\phi$, will usually denote scalar constants. 

This section contains the statements of some lemmas from \cite{collet} and proofs of some other lemmas. Of note is Proposition \ref{prop.inherit} because it allows us to induce to the product system an important and desirable property of the base map $T$.

\begin{lemma} For any $k>0$ and any $u \in \R$
\begin{equation}
 \sum_{j = 1}^k \chi_{\{\PHI{j}\geq u\}} \geq \chi_{\{Z_k\geq u\}} \geq \sum_{j = 1}^k \chi_{\{\PHI{j}\geq u\}} - \sum_{l \neq j}\chi_{\{\PHI{j}\geq u\}}\chi_{\{\PHI{l}\geq u\}}\label{lemma1}
\end{equation}
\end{lemma}

\begin{lemma}\label{3.2}
 For any integers $r$ and $k$ $\geq 0$, 
\begin{equation}
 0 \leq \prob(Z_r<u) - \prob(Z_{r+k}<u) \leq k\prob(\PHI{0}\geq u) \label{lemma2}
\end{equation}
\end{lemma}

\begin{lemma}
 For any positive integers m, p and t, 
\begin{equation}
\begin{split}
\left|
 \prob(Z_{m+p+t}<u) - \prob(Z_m<u)+\sum_{j = 1}^p\int\chi_{\{\PHI{0}\geq u\}}\chi_{\{Z_m<u\}}\circ f^{p+t-j} \text{d}\nu\right|\\
\leq 2p\sum_{j =1}^p\int\chi_{\{\PHI{0}\geq u\}}\chi_{\{\PHI{0}\geq u\}}\circ f^j \text{d}\nu+ t\prob(\PHI{0}\geq u)
\end{split}\label{lemma2b}
\end{equation}

\end{lemma}

The proofs for these lemmas can be found in \cite{collet}.

\begin{prop}\label{prop.inherit}
Let $\mu_X$ be the invariant, ergodic measure with respect to the map $T:X\to X$. Suppose $$E^X_n:=\left\{x\in X: d(T^jx,x)<\frac{1}{n}\mbox{ for some }j\le g(n)\right\}$$ satisfies $\mu_X(E_n^X)\le\frac{C}{n^\beta}$ for some constant $C>0$ and some $\beta>0$. Then, under the hypotheses of Theorem \ref{thm.main}, $\nu(\tilde E_n)\le\frac{C}{n^{\beta}}$ where $$\tilde E_n = \left\{(x, \theta)\in X\times Y: d(f^j(x, \theta), (x, \theta))<\frac{1}{n}\mbox{ for some }j\le g(n)\right\}.$$
\end{prop}

\begin{proof}

$(x, \theta)\in \tilde E_n$ implies $d(f^j(x, \theta), (x, \theta))<\frac{1}{n}$ for some $j\le g(n)$ and so $$\sqrt{d_X(T^jx,x)^2+d_Y(u^j(x, \theta),\theta)^2}<\frac{1}{n}$$ for such $j$ and so $d_X(T^jx, x)<\frac{1}{n}$. Thus, $x\in E_n^X$ and so $\tilde E_n\subset E_n^X\times Y$. %Therefore, $(\mu_X\times\lambda_Y)(\tilde E_n)\le\frac{C_1}{n^\beta}$. 

Define a new measure $\Delta$ on $X$ as $\Delta(A) := \nu(A\times Y)$. If $\lambda_X(A) = 0$ then $\mu_X(A) = 0$ and so $\mu_X\times \lambda_Y(A\times Y) = 0$ and thus $\nu(A\times Y) = 0$. Therefore, $\Delta$ is absolutely continuous with respect to the Lebesgue measure on $X$. Further, 
\begin{eqnarray*}
f^{-1}(A\times Y) & = & \left\{(x, \theta)| (Tx, u(x,\theta))\in A\times Y\right\}\\
& = & \left\{x\in T^{-1}A, (x, \theta)\in u^{-1}Y\right\}\\
& = & \left\{x\in (T^{-1}A\cap X), \theta\in Y\right\}\\
& = & T^{-1}(A) \times Y
\end{eqnarray*}
and so $\nu(f^{-1}(A\times Y)) = \nu(T^{-1}(A) \times Y)$. Therefore $$\Delta(T^{-1}A) = \nu(T^{-1}A\times Y) = \nu(f^{-1}(A\times Y)) = \nu(A\times Y) = \Delta(A).$$

To prove that $\Delta$ is ergodic for $T$, if $T^{-1}A = A$ then $\mu_X(A) = 0$ or 1 from which it follows that $\mu_X \times\lambda_Y(A\times Y) =0 \mbox{ or } 1$. Therefore by redefining $H$ (recall that $H$ is the density of $\nu$) on a $\mu_X\times \lambda_Y$ measure 0 set if necessary we have $$\nu(A\times Y) = \int_{A\times Y}H d(\mu_X\times \lambda_Y) = 0 \mbox{ or } 1.$$ Therefore, $\Delta(A) = 0$ or $1$.

Since the measures on $X$ are absolutely continuous with respect to Lebesgue, and hence unique, $\Delta(A) = \mu_X(A)$ from where it follows that $$\nu(\tilde E_n) \le \nu (E_n^X\times Y) = \Delta(E_n^X) = \mu(E_n^X)\le\frac{C}{n^\beta}.$$

\end{proof}

\begin{lemma} Under the assumptions of Theorem \ref{thm.main}, for $\nu$ $a.e$ $(x_0, \theta_0)\in X\times Y$  
 \begin{equation}
  n\sum_{j = 1}^{n^{\gamma'}}\nu(\PHI{0}> u_{n}, \PHI{j}>u_{n})\rightarrow 0 \mbox{ as } n\rightarrow \infty.
 \end{equation} \label{dprimeun}
 %where $\psi$ satisfying $0<D\gamma' <\psi<\beta$ may be chosen arbitrarily.
\end{lemma}

\begin{proof}
We begin by recalling that $H\in\L^{1+\delta}(\lambda)\subset \L^{1}(\lambda)$. Let $$E_n = \{(x, \theta):d(f^j(x, \theta), (x, \theta))<\frac{1}{n} \mbox{ for some } j\le g(n) \}$$ where $g(n)$ is as in Theorem \ref{thm.main}.  Let $D\gamma'<\psi<\beta$ and $\delta>0$. Define $$L_n(x, \theta) := \sup_{r>0}\frac{1}{\lambda(B_r(x, \theta))}\int_{B_r(x,\theta)}H\chi_{E_n}d\lambda.$$ 

By the Hardy-Littlewood Maximal Principle, since $H\chi_{E_n} \in \L^1(\lambda)$,$$\lambda(L_n(x, \theta)>\delta)\le\frac{C}{\delta}\|H\chi_{E_n}\|_1\le\frac{C}{\delta}\nu(E_n)\le\frac{C}{\delta n^\beta}$$  Choose $\gamma$ such that $\gamma(\beta -\psi)>1$. Replacing $\delta$ by $\frac{1}{n^{\gamma\psi}}$ and $n$ by $n^\gamma$ we get $$\lambda(L_{n^\gamma}>\frac{1}{n^{\psi\gamma}})\le\frac{C}{n^{\gamma(\beta-\psi)}}.$$ Therefore we have $$\sum_n\lambda(L_{n^\gamma}>\frac{1}{n^{\gamma\psi}})\le\sum_n\frac{C}{n^{\gamma(\beta-\psi)}}$$ which is summable. Hence, by the Borel-Cantelli Lemma, for $\lambda$ a.e. $(x_0, \theta_0)\in X\times Y$, we have $(x_0, \theta_0) \notin \limsup\{L_{n^\gamma}>\frac{1}{n^{\gamma\psi}}\}$ and so there exists $N(x_0, \theta_0)$ such that $n\ge N(x_0, \theta_0)\implies L_{n^\gamma}\le\frac{1}{n^{\psi\gamma}} $, i.e., $$\sup_{r>0}\frac{1}{\lambda(B_r(x_0, \theta_0))}\int_{B_{r}(x_0, \theta_0)}H\chi_{E_{n^\gamma}}d\lambda\le\frac{1}{n^{\gamma\psi}}.$$ Set $r = \frac{1}{n^{\gamma}}$ in the above to get $$n^{\gamma D}\int_{B_{\frac{1}{n^{\gamma}}}(x_0, \theta_0)}H\chi_{E_{n^\gamma}}d\lambda\le\frac{1}{n^{\psi\gamma}}.$$ Therefore we have
\begin{equation}
\nu\left\{\left\{d((x, \theta), (x_0, \theta_0))<\frac{1}{n^\gamma}\right\}\cap E_{n^\gamma}\right\}\le\frac{1}{n^{\psi\gamma+\gamma D}}.\label{intersection}
\end{equation}

Let $\tilde g(n)$ be a function that is increasing in $n$ with the property $g\left(\frac{n}{2}\right)\le \tilde g(n)\le\tilde g(2n)\le g(n)$. Let $k = \left(\frac{n^{1/D}}{2e^{-v}}\right)^{\frac{1}{\gamma}}$. Then we have,
\begin{eqnarray}
\lefteqn{\left\{(x, \theta):d((x, \theta), (x_0, \theta_0))\le\frac{e^{-v}}{n^{1/D}}, d(f^j(x, \theta), (x_0, \theta_0))\le\frac{e^{-v}}{n^{1/D}}\mbox{ for some } j\le \tilde g(n^{\frac{1}{D}}/e^{-v})\right\}}& &\nonumber\\
&\subset &\left\{(x, \theta):d((x, \theta),(x_0, \theta_0))\le\frac{e^{-v}}{n^{1/D}}, d(f^j (x, \theta), (x, \theta))<\frac{2e^{-v}}{n^{1/D}}\mbox{ for some } j\le \tilde g(\frac{n^{1/D}}{e^{-v}})\right\}\nonumber\\
&\subset& \left\{(x, \theta):d((x, \theta), (x_0, \theta_0))<\frac{2e^{-v}}{n^{1/D}}, d(f^j(x, \theta), (x, \theta))<\frac{2e^{-v}}{n^{1/D}}\mbox{ for some } j\le\tilde g(\frac{n^{1/D}}{e^{-v}})\right\}\nonumber\\
&\subset& \left\{(x, \theta):d((x, \theta), (x_0, \theta_0))<\frac{1}{k^\gamma}, d(f^j(x, \theta), (x, \theta))<\frac{1}{k^{\gamma}}\mbox{ for some }j\le\tilde g(2k^\gamma)\right\}\nonumber\\
&\subset& \left\{(x, \theta):d((x, \theta), (x_0, \theta_0))<\frac{1}{k^\gamma}, d(f^j(x, \theta), (x, \theta))<\frac{1}{k^{\gamma}}\mbox{ for some }j\le g(k^\gamma)\right\}\label{intersection2}
\end{eqnarray}
so that, by \eqref{intersection} and \eqref{intersection2}, for any $j\le g\left(\frac{n^{\frac{1}{D}}}{2e^{-v}}\right)$
$$\nu \{\PHI{0}>u_n, \PHI{j}>u_n\}\le\frac{(2e^{-v})^{\psi+D}}{n^{1+\frac{\psi}{D}}}.$$ Therefore, 
\begin{equation}
n\sum_{j = 1}^{g\left(\frac{n^{1/D}}{2e^{-v}}\right)}\nu\{\PHI{0}>u_n, \PHI{j}>u_n\}\to 0 \iff \frac{g\left(\frac{n^{1/D}}{2e^{-v}}\right)}{n^{\psi/D}} \to 0.
\end{equation}

Since $\psi>D\gamma'$ we get the above result.
\end{proof}

\begin{lemma}\label{decaycorr}
 Let $B_{r}(x, \theta)$ be a ball of radius $r$ and let $\epsilon>0$ be arbitrary. Let $\kappa$ be conjugate to $1+\delta$ ( i.e, $\frac{1}{1+\delta}+\frac{1}{\kappa} = 1$) and let $A$ be any measurable set. Then, under the assumptions of Theorem \ref{thm.main}, there exist constants $C_1$ and $C_2$ so that
\begin{equation}
 \abs{\nu(B_r\cap f^{-t}(A)) - \nu(B_r)\nu(A)} \le C_1\|H\|_{1+\delta}^{\lambda, (x, \theta)}\left(\nu(A)+1\right)r^{\frac{D+\epsilon}{\kappa}}+\frac{C_2}{r^{1+\epsilon}t^{\alpha}}
\end{equation}

\end{lemma}

\begin{proof}
We construct a H\"older continuous approximation to the characteristic function for $B_r$. Let $r' = r-r^{1+\epsilon}$. Construct $\Phi_B$ by letting it be 1 on the inside of the ball of radius $r'$ around $(x, \theta)$ and letting it decay to 0 at a linear rate  between $r$ and $r'$ .  The Lipschitz constant of this function may be chosen to be $\frac{1}{r^{1+\epsilon}}$. 
%An example of a function  in one dimension  decaying at the rate $\hat\alpha$ inside $[0, r]$ may be seen to be the function $\frac{r^{\hat\alpha} - x^{\hat\alpha}}{r^{\hat\alpha}}$. 
%For $\hat\alpha\ge 1$, this function has a $\hat\alpha$-H\"older norm proportional to $\frac{1}{r^{\hat\alpha}}$

Next, we note that $\lambda(B_r\setminus B_r') = r^D - (r-r^{1+\epsilon})^D \le 2^Dr^{D+\epsilon}$ and so we have, \begin{equation}\begin{split}\|\Phi_B-\chi_{B_r}\|_1^\nu = \int\abs{\Phi_B-\chi_{B_r}}d\nu\le\nu(B_r\setminus B_{r'}) = \int H\chi_{B_r\setminus B_{r'}}d\lambda \\ \le  \|H\|_{1+\delta}^{\lambda, (x, \theta)} \|\chi_{B_r\setminus B_{r'}}\|_{\kappa}^{\lambda}  \le C_1\|H\|_{1+\delta}^{\lambda, (x, \theta)}r^{\frac{D+\epsilon}{\kappa}}. \end{split}\label{est}\end{equation}

Finally,
\begin{eqnarray}
\lefteqn{\abs{\int\chi_B\chi_A\circ f^{t}d\nu - \int\chi_B d\nu\int\chi_A d\nu}} \nonumber\\
&\le & \abs{\int\chi_B\chi_A\circ f^t d\nu - \Phi_B\chi_A\circ f^t d\nu}\nonumber\\
&\quad& + \abs{\int\Phi_B\chi_A\circ f^t d\nu - \int\Phi_B d\nu\int\chi_A d\nu}\nonumber\\
&\quad& + \abs{\int\Phi_B d\nu\int\chi_A d\nu - \int\chi_A d\nu\int\chi_B d\nu}\nonumber\\
&\le&  \|\chi_A\circ f^t\|_\infty \|\chi_B-\Phi_B\|^\nu_1+\frac{C_2\|\chi_A\|_\infty\|\Phi_B\|_{\hat\alpha}}{t^{\alpha}}+\nu(A)\|\chi_B - \Phi_B\|_1^\nu.
\end{eqnarray}
A substitution of estimates from equation \eqref{est} completes the proof.
\end{proof}

\section{Proof of Theorem}

To prove Theorem \ref{thm.main}, we begin by breaking $n$ as a product of $p$ and $q$ with $p=\sqrt{n}$. We note that $$\nu(Z_n<u_n)\approx \nu(Z_{n+qt}<u_n)$$ where $t$ is a monotonically increasing function chosen  to satisfy $\frac{t}{p}\to 0$. The main estimate in the proof is $$\nu(Z_{n+qt}<u_n)\approx (1-p\nu(\PHI{0}\ge u_n))^q.$$ The function $t$ needs to be chosen so that terms of the form $n\sum_{j=1}^p\nu(\PHI{0}\ge u_n, \PHI{j}\ge u_n)$ that appear in the error to the above approximation can be broken into sums over $1\le j\le t$ and $t<j\le p$ with $t$ being small enough for growth of terms in the first sum to be killed by Lemma \ref{dprimeun} while large enough for growth in the second sum to be killed by Lemma \ref{decaycorr}.

\begin{thm}
Under the hypotheses of Theorem \ref{thm.main}, for $\nu$ $a.e.(x, \theta)$ and for any $v\in\R$, 
\begin{equation}
\lim_{n\rightarrow\infty}\nu\left(Z_n^{(x, \theta)}<u_n\right) = e^{H(x, \theta)e^{-Dv}}.
\end{equation}
\end{thm}

\begin{proof}
Choose $(x, \theta)\notin \limsup_{n\rightarrow\infty}E_{n}$ such that $$\lim_{a\rightarrow 0}\frac{1}{\lambda(B_{a}(x, \theta))}\nu(B_{a}(x, \theta)) = H(x, \theta).$$

%Let $1>\epsilon>0$ be chosen and assume $n$ is large enough (how large to be specified later). 

%Define $p = n^{\epsilon/(2D)}$ and $q =n^{1-\epsilon/(2D)}$. 
Then from above $$\lim_{n\rightarrow \infty} n\nu(B_{\frac{e^{-v}}{n^{1/D}}}(x, \theta)) = e^{-Dv}H(x, \theta).$$ 

Choose \begin{equation}\label{epschoice}\epsilon>D\kappa\left(\frac{3}{2} - \frac{1}{\kappa}\right)\end{equation} and $0<\tau<\min\{\gamma', \frac{1}{2}\}$ such that \begin{equation}\label{alphacond}\alpha > \frac{\frac{1+\epsilon}{D}+\frac{3}{2}}{\tau} > \frac{\frac{1}{D}\left(1+D\kappa\left(\frac{3}{2}-\frac{1}{\kappa}\right)\right)+\frac{3}{2}}{\min\left\{\gamma', \frac{1}{2}\right\}}\end{equation}

Define $t = n^{\tau}$, $p = \sqrt{n}$ and $q = \sqrt{n}$. Note that, by Lemma \ref{3.2}, $$\left|\prob(Z_n<u_n) - \prob(Z_{q(p+t)}<u_n)\right|\le   qt \prob(\PHI{0}\ge u_n).$$ 
Now, for $1\le l\le q$
\begin{eqnarray}
\lefteqn{\abs{\prob(Z_{l(p+t)}<u_n) - (1-p\prob(\PHI{0}\ge u_n))\prob(Z_{(l-1)(p+t)}<u_n)}} & &\nonumber\\
& = & \abs{p\prob(\PHI{0}\ge u_n)\prob(Z_{(l-1)(p+t)}<u_n) + \prob(Z_{l(p+t)}<u_n) - \prob(Z_{(l-1)(p+t)}<u_n)}\nonumber\\
&\le& \abs{p\prob(\PHI{0}\ge u_n)\prob(Z_{(l-1)(p+t)}<u_n) - \sum_{j = 1}^p\E{\chi_{\left\{\PHI{j}\ge u_n\right\}}\chi_{\left\{Z_{(l-1)(p+t)}<u_n\right\}}\circ f^{p+t}}}\nonumber\\
& \quad & +\abs{\prob(Z_{l(p+t)}<u_n) - \prob(Z_{(l - 1)(p+t)}<u_n) + \sum_{j = 1}^p\E{\chi_{\{\PHI{j}\ge u_n\}}\chi_{\{Z_{(l - 1)(p+t)}<u_n\}}\circ f^{p+t}}}\nonumber\\
& = & \abs{p\prob(\PHI{0}\ge u_n)\prob(Z_{(l-1)(p+t)}<u_n) - \sum_{j = 1}^p\E{\chi_{\{\PHI{j}\ge u_n\}}\chi_{\{Z_{(l-1)(p+t)}<u_n\}}\circ f^{p+t}}}\nonumber\\
& \quad & +\abs{\prob(Z_{lp+lt}<u_n) - \prob(Z_{lp+lt-(p+t)}<u_n) + \sum_{j = 1}^p\E{\chi_{\{\PHI{j}\ge u\}}\chi_{\{Z_{lp+lt - (p +t)}<u_n\}}\circ f^{p+t}}}\nonumber\\
\end{eqnarray}

By Lemma \ref{lemma2b} we have
\begin{equation}
\begin{split}
 \abs{\prob(Z_{lp+lt}<u_n)-\prob(Z_{(l-1)(p+t)}<u_n)+\sum_{j = 1}^p\E{\chi_{\{\PHI{0}\ge u_n\}}\circ f^j \chi_{\{Z_{(l-1)(p+t)}\}}\circ f^{p+t}}}\\
\le 2p\sum_{j = 1}^p\E{\chi_{\{\PHI{0}\ge u_n\}}\chi_{\{\PHI{0}\ge u_n\}}\circ f^j} + t\prob(\PHI{0}\ge u_n)\\
\end{split}
\end{equation}

For the remaining part, 
\begin{eqnarray}
 \lefteqn{\abs{p\prob(\PHI{0}\ge u_n)\prob(Z_{(l-1)(p+t)}<u_n) - \sum_{j = 1}^p\E{\chi_{\{\PHI{j}\ge u_n\}}\chi_{\{Z_{(l-1)(p+t)}<u_n\}}\circ f^{p+t}}}}& &\nonumber\\
&\le& \sum_{j = 1}^p\abs{\prob(\PHI{0}\ge u_n)\prob(Z_{(l-1)(p+t)}<u_n) - \E{\chi_{\{\PHI{j}\ge u_n\}}\chi_{\{Z_{(l-1)(p+t)}<u_n\}}\circ f^{p+t}}}\nonumber\\
&\le& pC_1\frac{e^{-v\frac{D+\epsilon}{\kappa}}}{n^{\frac{D+\epsilon}{D\kappa}}} + p\frac{C_2n^{\frac{1+\epsilon}{D}}}{e^{-v(1+\epsilon)}t^\alpha}
\end{eqnarray}
for large $n$ by Lemma \ref{decaycorr}.

Define 
\begin{eqnarray*}
\lefteqn{\Gamma_n}\hspace{.2cm} &:=&t\prob(\Phi\circ f^0\ge u_n) +2p\sum_{j = 1}^p\E{\chi_{\{\PHI{0}\ge u_n\}}\chi_{\{\PHI{0}\ge u_n\}}\circ f^j}
 + pC_1\frac{e^{-v\frac{D+\epsilon}{\kappa}}}{n^{\frac{D+\epsilon}{D\kappa}}} + p\frac{C_2n^{\frac{1+\epsilon}{D}}}{e^{-v(1+\epsilon)}t^\alpha}
\end{eqnarray*}

Therefore we have, for $1\le l \le q$
$$\abs{\prob(Z_{l(p+t)}<u_n) - (1-p\prob(\PHI{0}\ge u_n))\prob(Z_{(l-1)(p+t)}<u_n)}\le \Gamma_n.$$ Since $n\nu(\PHI{0}\ge u_n) \to e^{-Dv}H(x, \theta)$, for $n$ large enough, $p\prob(\PHI{0}\ge u_n)<1$, and so on applying the above formula inductively we get
$$\abs{\prob(Z_{q(p+t)}<u_n) - (1-p\prob(\PHI{0}\ge u_n))^q} \le q\Gamma_n+\frac{C_3 \|H\|_{1+\delta}\left(1-p\nu\left(\Phi\ge u_n\right)\right)^q}{n^\frac{1}{\kappa}}.$$

We now show that $q\Gamma_n\rightarrow 0$ as $n\rightarrow \infty$ and this will complete the proof because $$\left(1-\frac{pq\nu(\PHI{0}\ge  u_n)}{q}\right)^q \to e^{e^{-Dv}H(x, \theta)}.$$

By Lebesgue's Differentiation Theorem, for $\nu$ a.e. $(x, \theta)$ $$n\prob(\PHI{0}\ge u_n)\rightarrow e^{-Dv}H(x, \theta)$$ and so since $$\frac{t}{p}\rightarrow 0 \mbox{ as } n\rightarrow \infty$$ we have $$\lim_{n\rightarrow 
\infty}qt\prob(\PHI{0}\ge u_n) = 0$$ 

Also, $$nC_1\frac{e^{-v\frac{D+\epsilon}{\kappa}}}{n^{\frac{D+\epsilon}{D\kappa}}}\to 0$$ because $\epsilon > \frac{3D\kappa}{2} - D$.
Further, $$n\frac{C_2n^{\frac{1+\epsilon}{D}}}{e^{-v(1+\epsilon)}t^\alpha}\to 0\mbox{ because } \alpha>\frac{\frac{3}{2}+\frac{1}{D}(1+\epsilon)}{\tau} $$ by equation \eqref{alphacond}.

For the remaining part
\begin{equation}\label{qpsum}
\begin{split}
qp\sum_{j = t}^p\prob(\{\PHI{0}\ge u_n\}\cap f^{-j}\{\PHI{0}\ge u_n\})\le qp^2\prob(\PHI{0}\ge u_n)^2+qp^2C_1\frac{e^{-v\frac{D+\epsilon}{\kappa}}}{n^{\frac{D+\epsilon}{D\kappa}}}\\ +qp^2\frac{C_2n^{\frac{1+\epsilon}{D}}}{e^{-v(1+\epsilon)}t^\alpha}\end{split}
\end{equation}

We show that the terms on the right hand side converge to 0 as $n \to \infty$. Since $qp\prob(\PHI{0}\ge u_n)\to e^{-Dv}H(x, \theta)$, $$qp^2\prob(\PHI{0}\ge u_n)^2 \sim \frac{e^{-2Dv}H(x, \theta)^2}{q} \to 0 \mbox{ as } q\to\infty$$ Next, by \eqref{epschoice}, $$qp^2C_1\frac{e^{-v\frac{D+\epsilon}{\kappa}}}{n^{\frac{D+\epsilon}{D\kappa}}}\sim\frac{1}{n^{\frac{3}{2} - \frac{D+\epsilon}{D\kappa}}}\to 0$$ And, further, $$qp^2\frac{C_2n^{\frac{1+\epsilon}{D}}}{e^{-v(1+\epsilon)}t^\alpha}\sim \frac{1}{n^{\tau\alpha - \frac{3}{2}+\frac{1+\epsilon}{D}}}\to 0$$ Also, from Lemma \ref{dprimeun},
\begin{equation}
qp\sum_{j=1}^t\nu(\PHI{0}>u_n, \PHI{j}>u_n)\to 0 \mbox{ because } t = n^\tau \text{ and }\tau \le \gamma'.
\end{equation}
This completes the proof.
\end{proof}

%\begin{prop}{\bf Speed of Decay of Correlations}\\
%If Axiom 1, Axiom 2 and Axiom 3 hold, then a necessary condition for Theorem 1 to hold is that $$\alpha>\frac{1+\beta+\frac{\hat\alpha(1+2\beta)}{D}}{\beta}$$
%\end{prop}
%\begin{proof}
%From the proof above, it may be seen that  we need $$\alpha>\frac{1+\epsilon/2+\frac{\hat\alpha(1+\epsilon)}{D}}{\tau}$$ where $\tau<\min\{\epsilon/2, \gamma'\}$ where $\gamma'<\beta$. Hence the denominator in the above equation is always smaller than $\min\{\epsilon/2, \beta\}$ from where it can be seen that it is most advantageous to take $\epsilon = 2\beta$. This give us the lower bound $\alpha>\frac{1+\beta+\frac{\hat\alpha(1+2\beta)}{D}}{\beta}$
%\end{proof}

\section{Applications and Examples}\label{sec-applications}

We now verify the conditions of Theorem~\ref{thm.main} and hence establish corollaries \ref{cor.res} and \ref{cor.gouezel.map}. We will also discuss briefly two other important classes of maps: extensions to  the Manneville-Pommeau type maps and the Viana type maps. In the course of the discussion we will sketch why these maps satifsy all but one of the hypotheses of Theorem \ref{thm.main}.

\subsection{Uniformly and non-uniformly expanding maps of an interval modeled by Young towers}
\subsubsection{Piecewise $C^2$ uniformly expanding maps of the interval}
We suppose that $T: I\to I$ is a piecewise $C^2$ map of  an interval $I$ onto itself in the sense that there is a finite partition $\{I_i\}$ of the 
interval $I$, $T$ is $C^2$ on the interior of each $I_j$, $T:I_j\to I$ is onto and monotone, and  $|T^{'} (x)|>1+\delta$ for all $x$ lying in the interior of each $I_j$. It is known from ~\cite{Gora-Boyarsky}
that such maps possess an absolutely continuous mixing invariant measure $\mu$ and there exists a $C$ such that $\frac{1}{C} \le \frac{d\mu}{dx} \le C$. It is easy to see that there exists a $C_2>0$
such that  $m\{ x: d(x,Tx)<\frac{1}{n} \} \le \frac{C_2}{n}$ for all 
$n$ and furthermore the same argument shows that $m\{ x: d(x,T^jx)<\frac{1}{n} \} \le \frac{C_2}{n}$ for all $j$. 

Such maps possess a Young Tower with exponential return time tails~\cite{Young}, hence,  as shown in ~\cite{gouezel} for a residual
set of  $S^1$ cocycles $h: I\to S^1$ , the skew-extension $f$ of the base map $T$ has exponential  decay of correlations. Thus this class of maps satisfies the conditions of Theorem~\ref{thm.main}.

\subsubsection{Non-Uniformly Expanding Maps Modeled by a Young Tower}
Suppose $T:X\to X$ is a non-uniformly expanding map of an interval with bounded derivative, i.e.,$\sup_{x\in X} |T^{'} (x)|<C$, modeled
by a Young Tower with  exponential return time tails. Collet \cite{collet} has shown that there exists a $\beta>0$ for which $\mu(E_n^X)<\frac{C}{n^\beta}$ and so by Proposition \ref{prop.inherit} we may conclude that the system $f:X\times S^1\to X\times S^1$ defined by $f(x, \theta) = (Tx, \theta+h(x))$ for any measurable cocycle $h$ satisfies this property. Further, Gou\"ezel shows in \cite{gouezel} that for a residual set of H\"older cocycles, such systems satisfy the second hypothesis of Theorem \ref{thm.main} for an arbitrary $\alpha$ (by showing that decay is in fact exponential). Since the map $f$ along the group $S^1$ is an isometry, it's density with respect to the Lebesgue measure is 1 and hence the density of the invariant measure is just the density for $T$. Collet \cite{collet} shows that this density lies in $\L^p$ for some $p$ larger than 1, and so all the hypotheses of Theorem \ref{thm.main} are satisfied.

\subsection{Skew product with a curve of neutral points}\label{skew-prod}

We consider Gou\"ezel's map studied, for instance, in \cite{skew-neutral}. Define $F:S^1\to S^1$ by $F(\omega) = 4\omega$ and $T_\alpha:[0,1]\to [0,1]$ as
\begin{equation}
T_\alpha(x) = 
\begin{cases}
  x(1+2^\alpha x^\alpha) \text{\hspace{.4cm}if $0\le x\le \frac{1}{2}$}\\
2x-1 \text{ \hspace{1cm} if $\frac{1}{2}<x\le 1$}
 \end{cases}
\end{equation}
where $\alpha:S^1\to(0, 1)$ is a map with minimum $\alpha_{\min}$ and a maximum $\alpha_{\max}$ and satisfies
\begin{itemize}
 \item $\alpha$ is $C^2$
\item $0<\alpha_{\min}<\alpha_{\max}<1$
\item $\alpha$ takes the value $\alpha_{\min}$ at a unique point $x_0$ with $\alpha''(x_0)>0$
\item $\alpha_{\max}<\frac{3}{2}\alpha_{\min}$
\end{itemize}
The map $T:S^1\times[0,1]\to S^1\times[0, 1]$ is defined as $T(\omega, x) = (F(\omega), T_{\alpha(\omega)}(x))$. From \cite[Theorem 2.10]{skew-neutral}, the density $H$ of the map $T$ is $\L^{1}$ with respect to the product $\mu\times\mbox{Leb}$ where $\mu$ is the invariant measure on $S^1$ for $F$ (and is the same as the Lebesgue measure). Since $F$ is uniformly expanding we see that hypothesis (a) of Theorem \ref{thm.main} is satisfied. Further, from \cite{person-comm} 
\begin{equation}
\left|\int \Upsilon \Psi\circ T^n -\int \Upsilon \int \Psi \right|\le C n^{1-1/\alpha_{max}}\|\Upsilon\|_{\hat\alpha} \|\Psi\|_{\infty}
\end{equation}
and so hypothesis (b) is also satisfied. Further, by \cite[Theorem 2.10]{skew-neutral}, the density $H$ is Lipschitz on every compact subset of $S^1\times (0,1]$. The only places in the proof of Theorem \ref{thm.main} that we require the density to be in $\L^{1+\delta}$ is to estimate the volume of balls, and this requirement can be replaced by the Lipschitz requirement on every compact subset. 
Recall, that $B_r$ is a ball about a fixed point $(\omega, x)$ of radius $r$, and
\begin{equation*}
 \|\Phi_B-\chi_B\|_1^\nu\le \nu(B_r\setminus B_{r'}) = \int H\chi_{B_r\setminus B_{r'}}d\lambda.
\end{equation*}
Fix a closed ball $\Gamma$ with center $(\omega, x)$. For $r$ sufficiently small, $B_r\subset\Gamma$ and so $\|H\mid_\Gamma\|_{\infty}<\infty$. Therefore 
\begin{equation*}
 \int H\chi_{B_r\setminus B_{r'}} d\lambda \le \|H\mid_\Gamma\|_{\infty}\lambda(B_r\setminus B_{r'})
\end{equation*}
and so bounds of the type of Lemma \ref{decaycorr} may be obtained with $\kappa $ set equal to 1. Now, if we  choose $\epsilon>\frac{D}{2}$, in equation \eqref{qpsum}  we have $$qp^2\frac{e^{-v(D+\epsilon)}}{n^{1+\frac{\epsilon}{D}}} \to 0.$$ The last term in equation \eqref{qpsum} will converge to 0 if the function $\alpha$ is chosen so that $\alpha_{\max}$ satisfies $$\alpha_{\max}<\frac{\min\{\gamma', \frac{1}{2}\}}{\min\{\gamma', \frac{1}{2}\}+\frac{1}{D}\left(1+\frac{D}{2}\right)+\frac{3}{2}}.$$

\subsection{Some Other Extensions}

\subsubsection{The Viana Maps}

Let $T$ be a uniformly expanding map of the circle $S^1$ given by $T(\theta) = d\theta \mbox{ mod 1}$ for $d\ge 16$. Suppose $b:S^1\to S^1$ is a Morse function, that $u_\alpha(\theta, x) = a_0+\alpha b(\theta) - x^2$ and that $a_0$ is chosen so that $x=0$ is pre-periodic for $a_0 - x^2$. Let $\Phi_\alpha(\theta, x) = (T(\theta), u_\alpha(\theta, x))$. From \cite{alves}, for small enough $\alpha$, there is an interval $I \subset (-2, 2)$ for which $\Phi_\alpha(S^1\times I)\subset \mbox{ int }(S^1\times I)$.

Along the base, this map exhibits a uniformly expanding behavior, and thus, from Proposition \ref{prop.inherit}, we can conclude that the first hypothesis of Theorem \ref{thm.main} is satisfied. Also, it has been shown in \cite{gouezel-avmap} that such a system displays a decay of correlations at the rate of $O(e^{-c\sqrt{n}})$ which is faster than any polynomial. From \cite{alves}, we know that the density of the absolutely continuous invariant measure lies in $\L^1(\lambda)$. If we knew that this density was in $\L^{1+\delta}(\lambda)$ for small $\delta>0$, then all the hypotheses of Theorem \ref{thm.main} would be satisfied and in that case the limiting distribution obtained would be  $$\lim_{n\to\infty}\nu(Z_n^{(x, \theta)}<u_n) = e^{H(x, \theta)e^{-2v}}.$$

\subsubsection{Manneville-Pommeau type maps}

We will consider the Liverani-Saussol-Vaienti Map $T:[0, 1]\to [0, 1]$ defined as
\[T(x) = 
\begin{cases}
   x(1+2^{\omega}x^{\omega}) & x\in [0, \frac{1}{2})\\
2x-1 & x\in [\frac{1}{2}, 1]
  \end{cases}
\]

Near the origin, this map is $x\mapsto x+2^\omega x^{1+\omega}$ and the density near the origin is seen to be $h(x)\approx x^{-\omega}$ so $h\in \L^{\frac{1}{\omega}-\epsilon}$ for any $\epsilon>0$. It is a result from \cite{work_in_progress} that $$\mu_X\left\{x:d(T^jx, x)<\frac{1}{n} \text{ for some }0\le j\le g(n)\right\}\le \left(\frac{g(n)}{\sqrt{n}}\right)^{1-\omega}$$ so if we choose $u$ to be a cocycle, $g(n) = n^{\frac{1-\omega}{24}}$ and $\beta = \frac{1-\omega}{8}$, we see that for $Y = S^1$ we have $D = 2$, $\gamma' = \frac{1-\omega}{24}<\frac{\beta}{D}$ and $\mu_X(E_n^X)<\frac{C}{n^\beta}$. Further, since we have an isometry along the fiber, the density $H$ for $\nu$ will lie in $\L^{\frac{1}{\omega}-\epsilon}$ and so all the hypotheses of Theorem \ref{thm.main} are met except that the rate of decay of correlations for such an extension $f = (T, u)$ is not known. If a rate satisfying condition (b) can be established, we will be able to establish the extreme value law.\\

{\bf Acknowledgements:} The author would like to thank his PhD advisor, Matthew Nicol, for his encouragement, support and invaluable inputs. The author would also like to thank Mark Holland for helpful comments and references; S\'ebastien Gou\"ezel for useful discussions about his work, particularly, regarding decay of correlations; Jos\~e Alves and Dimitry Dolgopyat for helpful references to the literature; and the organizers of the Workshop in Chaotic Properties of Dynamical Systems at Warwick, in August of 2007. This research was undertaken as part of the author's PhD and was supported in part by NSF grant DMS-0607345.

\end{document}